\documentclass[11pt]{siamltex}

\usepackage[english]{babel}
\usepackage[dvips]{graphicx}
\usepackage{amssymb}
\usepackage{bm}
\usepackage{amsmath}
\usepackage{xcolor}
\usepackage{url}
\usepackage{hyperref}

\newcommand{\RR}{\mathbb{R}}

\newcommand{\PP}{\mathbb{P}}
\def\C{\mathcal{C}}

\title{On the error of best polynomial approximation of composite functions \thanks{This research has been accomplished within ``Research ITalian network on Approximation'' (RITA).}}

\author{
Luisa Fermo\thanks{Department of Mathematics and Computer Science, University of Cagliari,
Via Ospedale 72, 09124 Cagliari, Italy, email: fermo@unica.it. \thanks{Corresponding author}}
\and
Concetta Laurita \thanks{Department of Mathematics, Computer Science and Economics, University of Basilicata, Viale dell'Ateneo Lucano 10, 85100 Potenza, Italy, email: concetta.laurita@unibas.it, mariagrazia.russo@unibas.it}
\and
Maria Grazia Russo \footnotemark[3]
}

\begin{document}
\maketitle

\begin{abstract}
The purpose of the paper is to provide a characterization of the error of the best polynomial approximation of composite functions in weighted spaces. Such a characterization is essential for the convergence analysis of numerical methods applied to non-linear problems or for numerical approaches that make use of regularization techniques to cure low smoothness of the solution.
This result is obtained through  an estimate of the derivatives of composite functions in weighted uniform norm.
\end{abstract}

\section{Introduction}
Consider a fixed Jacobi weight
$$u(x)=v^{\gamma,\delta}(x):=(1-x)^\gamma(1+x)^\delta, \qquad x\in (-1,1), \quad \gamma,\delta\ge 0.$$
We denote by $C_u$ the space of all functions $f$ continuous on $(-1,1)$ and satisfying
\begin{equation*}
	\lim_{x\rightarrow +1}f(x)u(x)=0\quad \mbox{if $\gamma>0$},
	\qquad \mbox{and}
	\qquad
	\lim_{x\rightarrow -1}f(x)u(x)=0\quad \mbox{if $\delta>0$},
\end{equation*}
equipped with the norm
\[
\|f\|_{C_u}:=\|fu\|_\infty= \max_{x\in [-1,1]}|f(x)|u(x).
\]
It is well-known (see for instance \cite{mastromilobook}) that the Weierstrass approximation theorem holds true in this Banach space and we have
\[
f\in C_u\ \Longleftrightarrow \ \lim_{m\rightarrow \infty} E_m(f)_u=0,
\]
where $E_m(f)_u$  denotes the error of best approximation of $f\in C_u$ in the space $\PP_m$ of all algebraic polynomials of degree at most $m$, namely
\[
E_m(f)_u:=\inf_{P\in\PP_m}\|(f-P)\|_{C_u}.
\]
The rate of convergence of such error as $m\rightarrow\infty$ depends on the smoothness of the function $f$.  A renewed inequality used in estimating the best approximation is the following Favard type inequality
(see for instance \cite{mastromilobook} and the references therein)
\begin{equation}\label{favard}
	E_m(f)_u \leq \frac{\mathcal{C}}{m^r}\|f^{(r)}\varphi^r u\|_\infty
	\end{equation}
where here and in the following $\mathcal{C}$ is a positive absolute constant, $f^{(r)}$ denotes the $r$th derivative of $f$ and $\varphi(x)=\sqrt{1-x^2}$.

In different contexts in which the ``global approximation" (i.e. in processes involving the polynomial approximation) is used, it becomes crucial to estimate  the best polynomial approximation of composed functions. This situation obviously occurs in non linear problems, but also in other contexts like for instance in methods using regularization techniques (see, for instance, ~\cite{diogo2017},~\cite{FermoRusso},~\cite{galperin2000},
~\cite{MonScu1998},~\cite{Pedas2006bis}, and ~\cite{vainikko2008}).

Looking at the Favard inequality, in order to estimate $E_m(f \circ g)_u$ we need an upper bound for the weighted norm of the derivatives of the composite functions. A crucial step for our aim  is to provide, in weighted uniform norm, an estimate of the derivatives of a given function in terms of a fixed higher derivative of the same function. According to our knowledge, this is only known for functions belonging to  $C^{r}([a,b])$. In fact, in \cite[Lemma 2.1]{Ditzian} for a given function $f \in C^{r}([a,b])$ the following inequality was stated
\begin{equation} \label{EstDerDitzian}
\|f^{(k)}\|_\infty \leq \mathcal{C} \left( \frac{\|f\|_\infty}{(b-a)^k} +(b-a)^{r-k}\|f^{(r)}\|_\infty \right), \quad 0<k< r,
\end{equation}
where $\mathcal{C}$ is a positive constant independent of the interval $[a,b]$ and $f$.

In this paper, we first prove the analogous of (\ref{EstDerDitzian}) in the weighted case and then we give the estimate of the derivative of a composite function in a very general case, i.e. the case in which $f$ is a multivariate function and $g$ is a vector of functions.
We will present the results in $[-1,1]$, without loss of generality, since by linearity  analogous results can be deduced in the generic interval $[a,b]$ of $\mathbb{R}$.

\section{Sobolev-type spaces}
Let us introduce the weighted Sobolev--type space of order $1 \leq r \in \mathbb{N}$ \cite{mastromilobook}
\begin{equation*}
{W}^r_{u}=\left\{f \in C_{u}: f^{(r-1)} \in AC((-1,1)), \,  \|f^{(r)}\varphi^r u \|_\infty <\infty \right\},
\end{equation*}
where $\varphi(x)=\sqrt{1-x^2}$ and $AC((-1,1))$ denotes the set of all absolutely continuous functions on $(-1,1)$.
We equip $W^r_{u}$ with the norm
$$\|f\|_{W^r_{u}}:=\|fu\|_\infty+\|f^{(r)}\varphi^r u \|_\infty.$$

For multivariate functions $f:\Omega \rightarrow \RR$ with $\Omega$ open subset of $\RR^n$, we define the Sobolev space ${\bm W}^{r}(\Omega)$ \cite{brezisbook}
as the set of all functions $f$ in $\Omega$ such that for every $n$-tuples of nonnegative integers $\ell=(\ell_1,\ell_2,\dots, \ell_n)$, with $|\ell|=\sum_{i=1}^n \ell_i \leq r$, the mixed partial derivatives
$$D^{\ell} f = \frac{\partial^{\ell_1+\ell_2+\dots+\ell_n} f}{\partial x_1^{\ell_1} \partial x_2^{\ell_2}\dots \partial x_n^{\ell_n}}$$
exist and $\|D^{\ell} f\|_\infty <\infty$ .
We endow this space with the norm
\begin{equation}\label{SobolevNormMultiD}
\|f\|_{{\bm W}^{r}(\Omega)}=\|f\|_\infty+ \sum_{1 \leq |\ell| \leq r} \|D^{\ell} f\|_\infty.
\end{equation}
The spaces $(W^r_u, \| \cdot\|_{W^r_u})$ and $({\bm W}^{r}(\Omega), \| \cdot\|_{{\bm W}^r(\Omega)})$  are Banach spaces.
\section{The Faa di Bruno's formula}
Let $f: \Omega \subseteq \RR^n \rightarrow \RR$ and $g: A\subseteq \RR \rightarrow \RR^n$ with $g(x)=\left(g_1(x),g_2(x),\ldots,g_n(x)\right)$,  be functions such that the range of $g$ is contained in the domain of $f$ and for which all the necessary derivatives are defined. \newline
In the case $n=1$ the well-known F\'aa di Bruno formula states
\begin{equation}\label{faa1}
(f \circ g)^{(r)}(x)=\sum \frac {r!}{k_1!k_2!\ldots k_r!}f^{(k1+k2+\ldots+k_r)}(g(x))\prod_{i=1}^{r}\left( \frac{g^{(i)}(x)}{i!} \right)^{k_i}
\end{equation}
where the sum is over all $r$-tuples of nonnegative integers $\left(k_1,k_2,\ldots,k_r\right)$ such that $k_1+2k_2+\ldots+rk_r=r.$

Combining the terms with the same value of $k_1+k_2+\ldots+k_r=k$ and noticing that $k_j$ has to be zero for $j>r-k+1$ lead to the following equivalent but simpler formula
\begin{equation*}
(f \circ g)^{(r)}(x)=\sum_{k=1}^r f^{(k)}(g(x)) \, {\bf{B}}_{r,k}(g'(x), g''(x),\dots g^{(r-k+1)}(x))
\end{equation*}
where ${\mathbf{B}}_{r,k}(x_1, x_2,\dots x_{r-k+1})$ are the partial or incomplete exponential Bell polynomials defined as
\begin{equation*}
{\bf{B}}_{r,k}(x_1, x_2,\dots x_{r-k+1})= \sum \frac{r!}{k_1! k_2! \dots k_{r-k+1}!} \prod_{i=1}^{r-k+1} \left( \frac{x_i}{i!} \right)^{k_i},
\end{equation*}
with the sum  taken over all sequences  $\left(k_1,k_2,\ldots,k_{r-k+1}\right)$ of nonnegative integers such that these two conditions are satisfied
\begin{equation*}
\sum_{i=1}^{r-k+1} k_i =k, \qquad \sum_{i=1}^{r-k+1} i k_i =r.
\end{equation*}
The formula \eqref{faa1} for the $r$-th derivative of a composite function can be generalized in the case $n>1$ as  (see \cite{Mishkov})
\begin{align} \label{Faa_multidimensional}
& (f \circ g)^{(r)}(x) \nonumber \\ & =\displaystyle\sum_0 \sum_1\sum_2 \cdots \sum_r \displaystyle \frac{r!}{\prod_{i=1}^r\prod_{j=1}^n  q_{ij}!\prod_{i=1}^r(i!)^{k_i}}\frac{\partial^kf(g(x))}{\partial x_1^{p_1} \partial x_2^{p_2} \cdots \partial x_n^{p_n}} \nonumber\\
 & \times \displaystyle  \prod_{i=1}^r\left(g_1^{(i)}(x)\right)^{q_{i1}}\left(g_2^{(i)}(x)\right)^{q_{i2}}\cdots\left(g_n^{(i)}(x)\right)^{q_{in}}
\end{align}
where the respective sums are over all negative solutions of the Diophantine equations, as follows
\begin{align} \label{Diophantine eq}
&\sum_0  \rightarrow  k_1+2k_2+\ldots+rk_r  =  r,\nonumber \\ \\
&\sum_i  \rightarrow q_{i1}+q_{i2}+\ldots+q_{in}  =  k_i,  \quad i=1.2,\ldots,r, \nonumber
\end{align}
and
\begin{align} \label{Diophantine eq2}
 p_j&=q_{1j}+q_{2j}+\ldots+q_{rj}, \quad j=1,2,\ldots,n, \nonumber \\ \\
k&=p_1+p_2+\ldots+p_n=k_1+k_2+\ldots+k_r. \nonumber
\end{align}

\section{Main results}
First, let us prove the following lemma which gives an estimate of the derivatives of a function $f \in W^r_u$, generalizing \eqref{EstDerDitzian}.\\

\begin{lemma}
Let $u(x)=(1-x)^{\gamma}(1+x)^\delta$ be a Jacobi weight with $0 \leq \gamma,\delta <1$ and $f \in W^r_u$. Then, for $0<k<r$
\begin{equation} \label{SeminormEstimater}
\|f^{(k)} \varphi^k u\|_\infty \leq \mathcal{C} \left(\|f u\|_\infty+\|f^{(r)} \varphi^r u\|_\infty \right)
\end{equation}
where $\mathcal{C}=\mathcal{C}(r,k,\gamma,\delta)$ is a positive constant independent of $f$.
\end{lemma}
\begin{proof}
First we note that it is enough to prove that, for any $0<k<r$,
\begin{equation} \label{SeminormEstimatek}
\|f^{(k)} \varphi^k u\|_\infty \leq \mathcal{C} \left(\|f u\|_\infty+\|f^{(k+1)} \varphi^{k+1} u\|_\infty \right)
\end{equation}
with $\mathcal{C}=\mathcal{C}(k)$, since \eqref{SeminormEstimater} con be deduced from \eqref{SeminormEstimatek} by induction on $k$. \newline
Fix $-1 \leq x \leq 0$ and  $h=1/k$.  By using the Taylor formula with integral remainder, being $f^{(k)}$ locally absolutely continuous and consequently $f^{(k+1)}$ locally integrable, we can write
\begin{equation*}
f(x+jh)=\sum_{i=0}^k\frac{(jh)^i}{i!}f^{(i)}(x)+\frac{1}{k!}\int_{0}^{jh}(jh-t)^kf^{(k+1)}(x+t)dt
\end{equation*}
from which it follows
\begin{align*}
h^kf^{(k)}(x)& =\sum_{j=0}^k\binom{k}{j}(-1)^{k-j}f(x+jh) \\ & -\frac{1}{k!}\sum_{j=1}^k\binom{k}{j}(-1)^{k-j}\int_{0}^{jh}(jh-t)^kf^{(k+1)}(x+t)dt
\end{align*}
and, then,
\begin{eqnarray} \label{estimate1}
h^k |f^{(k)}(x)|\varphi^k(x)u(x)   & \leq & \sum_{j=0}^k\binom{k}{j}|f(x+jh)|\varphi^k(x)u(x) \nonumber\\
 &\hspace{-4cm} +&\hspace{-2cm}  \frac{1}{k!}\sum_{j=1}^k\binom{k}{j}\int_{0}^{jh}(j h-t)^k\left|f^{(k+1)}(x+t)\right|\varphi^k(x)u(x)dt.
\end{eqnarray}
Let us estimate the first term at the right-hand side. Since $\varphi^k(x) \leq 1$ and
\begin{equation*}
\frac{u(x)}{u(x+jh)} \leq \frac{(1-x)^{\gamma}}{(1-x-jh)^{\gamma}}\leq \frac{2^\gamma}{(1-jh)^{\gamma}},
\end{equation*}
being $-1 \leq x \leq 0$ and $h>0$, using $\displaystyle \binom{k}{j}\leq 2^k$ for any $j \in \{0,1,\ldots,k\}$, we obtain
\begin{eqnarray*}
\sum_{j=0}^k\binom{k}{j}|f(x+jh)|\varphi^k(x)u(x)  &\leq&  \|fu\|_{\infty}\sum_{j=0}^k\binom{k}{j}\frac{u(x)}{u(x+jh)}\varphi^k(x)  \nonumber \\
& \leq &  \|fu\|_{\infty}2^{\gamma+k}\sum_{j=0}^k \frac{1}{(1-jh)^{\gamma}} \nonumber\\
& \leq &   \|fu\|_{\infty}2^{\gamma+k}\int_0^k (1-zh)^{-\gamma}dz \\ & =& \|fu\|_{\infty}\frac{2^{\gamma+k}}{(1-\gamma)h}.
\end{eqnarray*}
In order to estimate the second sum in \eqref{estimate1}, we observe that for $-1 \leq x \leq 0$, $1 \leq j \leq k$, $h=1/k>0$ and $0 \leq t \leq jh$ the following inequalities hold true
\begin{align*}
\frac{\varphi^k(x)}{\varphi^{k+1}(x+t)} &\leq \frac {2^{\frac{k}{2}}}{t^{\frac{1}{2}}(jh-t)^{\frac{k}{2}+\frac{1}{2}}}, \\
\frac{u(x)}{u(x+t)} &\leq \frac{(1-x)^{\gamma}}{(1-x-t)^\gamma} \leq \frac{2^{\gamma}}{(1-t)^\gamma} \leq \frac{2^{\gamma}}{(jh-t)^{\gamma}}.
\end{align*}
Consequently, we deduce
\begin{eqnarray} \label{estimate3}
& & \frac{1}{k!}\sum_{j=1}^k\binom{k}{j}\int_{0}^{jh}(j h-t)^k\left|f^{(k+1)}(x+t)\right|\varphi^k(x)u(x)dt  \nonumber \\ &\leq & \frac{1}{k!}\left\|f^{(k+1)}\varphi^{k+1}u\right\|_{\infty}  \sum_{j=1}^k\binom{k}{j}\int_{0}^{jh}(j h-t)^k
\frac{\varphi^k(x)}{\varphi^{k+1}(x+t)}\frac{u(x)}{u(x+t)}dt \nonumber \\ & \leq & \frac{2^{\frac{k}{2}}}{k!}+\gamma\left\|f^{(k+1)}\varphi^{k+1}u\right\|_{\infty}\sum_{j=1}^k\binom{k}{j}\int_{0}^{jh}
(j h-t)^{\frac{k-1}{2}-\gamma}t^{-\frac{1}{2}}dt  \nonumber\\
&=& \frac{1}{k!}2^{\frac{k}{2}+\gamma}\left\|f^{(k+1)}\varphi^{k+1}u\right\|_{\infty}B\left(1-\gamma,\frac{1}{2}\right)
h^{\frac{k}{2}-\gamma}\sum_{j=1}^k\binom{k}{j}j^{\frac{k}{2}-\gamma},
\end{eqnarray}
where $B$ denotes the beta function. 

Now, we note that for any integer $k \geq 1$, it is $$\sum_{j=1}^k\binom{k}{j}j^{\frac{k}{2}-\gamma} \leq 2^k k^{\frac{k}{2}-\gamma},$$
from which, combining \eqref{estimate1}-\eqref{estimate3}, we get
\begin{eqnarray} \label{estimatenegativex}
& & \max_{-1 \leq x \leq 0}\left|f^{(k)}(x)\varphi^{k}(x)u(x)\right| \nonumber  \\ & \leq & \frac{1}{h^k}\left[\frac{2^{\gamma+k}}{(1-\gamma)h}\|fu\|_{\infty}+
\frac{1}{k!}2^{\frac{k}{2}+\gamma}B\left(1-\gamma,\frac{1}{2}\right)h^{\frac{k}{2}-\gamma}2^k k^{\frac{k}{2}-\gamma}\left\|f^{(k+1)}\varphi^{k+1}u\right\|_{\infty}\right] \nonumber \\
&=& \frac{2^{\gamma+k}}{1-\gamma}k^{k+1}\|fu\|_{\infty}+\frac{k^k}{k!}2^{\frac{3}{2}k+\gamma}B\left(1-\gamma,\frac{1}{2}\right)
\left\|f^{(k+1)}\varphi^{k+1}u\right\|_{\infty}.
\end{eqnarray}
The case $0 \leq x \leq 1$ and $h=-1/k$ can be treated with similar arguments and leads to the following estimate
\begin{eqnarray} \label{estimatepositivex}
\max_{0 \leq x \leq 1}\left|f^{(k)}(x)\varphi^{k}(x)u(x)\right|
&\leq & \frac{2^{\delta+k}}{1-\delta}k^{k+1}\|fu\|_{\infty} \nonumber \\ &+& \frac{k^k}{k!}2^{\frac{3}{2}k+\delta}B\left(1-\delta,\frac{1}{2}\right)
\left\|f^{(k+1)}\varphi^{k+1}u\right\|_{\infty}.
\end{eqnarray}
Finally, combining \eqref{estimatenegativex} with \eqref{estimatepositivex}, we can deduce \eqref{SeminormEstimatek} with
the constant $\mathcal{C}=\mathcal{C}(k)$ given by
\begin{align*}
\mathcal{C}(k) & =\max\left\{\frac{2^{\gamma+k}}{1-\gamma}k^{k+1}, \frac{2^{\delta+k}}{1-\delta}k^{k+1}, \frac{k^k}{k!}2^{\frac{3}{2}k+\gamma}B\left(1-\gamma,\frac{1}{2}\right), \right. \\ & \left.  \frac{k^k}{k!}2^{\frac{3}{2}k+\delta}B\left(1-\delta,\frac{1}{2}\right)
\right\}.
\end{align*}
(The exact constant which we arrived at are not important since they are not the best possible.)
\end{proof}

The previous lemma is crucial to prove our first main result.
\begin{theorem}
Let $f:\Omega \rightarrow \RR$ with $\Omega$ open subset of $\RR^n$ and $g:(-1,1) \rightarrow \RR^n$ such that $\mathrm{Im}(g) \subseteq \Omega$. Moreover, let $u(x)=(1-x)^{\gamma}(1+x)^\delta$ be a Jacobi weight with $0 \leq \gamma,\delta <1$. If we assume that $f \in  {\bm W}^r(\Omega)$ and $g=\left(g_1,g_2,\ldots,g_n\right)$ with $g_j \in W^r_u$, for $j=1,2,\ldots,n$, then $f\circ g \in W^r_{u^r}$ and
\begin{equation} \label{EstDerCompFun}
\displaystyle  \|(f \circ g)^{(r)} \varphi^ru^r\|_\infty \leq \mathcal{C} n^r B_r \|f\|_{{\bm W}^{r}(\Omega)}  \prod_{j=1}^n \|g_j\|_{W^r_u}^{s_j}
\end{equation}
where $B_r$ is the $r$-th Bell number,  $\C=\C(r,\gamma,\delta)$ is a positive constant independent of $f$ and $g$, and
\begin{equation} \label{exponentp}
s_j=\left\{\begin{array}{lcr} 0, & \quad \mathrm{if}  \, & \|g_j\|_{W^r_u} \leq 1 \vspace{0.1cm} \\

r, & \quad  \mathrm{if}  \,& \|g_j\|_{W^r_u} >1  \end{array} \right., \qquad j=1,2,\ldots,n.
\end{equation}
\end{theorem}
\begin{proof}
Faa di Bruno's formula \eqref{Faa_multidimensional}, with \eqref{Diophantine eq}-\eqref{Diophantine eq2}, yields
\begin{align*}
 &   \left|(f \circ g)^{(r)}(x)\varphi^r(x)u^r(x)\right| \\ & \leq\sum_0 u^{r-k}(x)\sum_1\sum_2 \cdots \sum_r
 \frac{r!}{\prod_{i=1}^r\prod_{j=1}^n  q_{ij}!}\left|\frac{\partial^kf(g(x))}{\partial x_1^{p_1} \partial x_2^{p_2} \cdots \partial x_n^{p_n}}\right| u^{p_1+\ldots+p_n}(x) \\
& \times  \prod_{i=1}^r\frac{\left|g_1^{(i)}(x)\varphi^i(x)\right|^{q_{i1}}\left|g_2^{(i)}(x)\varphi^i(x)\right|^{q_{i2}}\cdots\left|g_n^{(i)}(x)
\varphi^i(x)\right|^{q_{in}}}{(i!)^{k_i}} \\
&  \leq   \C \sum_0 \sum_1\sum_2 \cdots \sum_r
\displaystyle \frac{r!}{\prod_{i=1}^r\prod_{j=1}^n  q_{ij}!}  \left|\frac{\partial^kf(g(x))}{\partial x_1^{p_1} \partial x_2^{p_2} \cdots \partial x_n^{p_n}}\right|
 \\ & \times \prod_{i=1}^r\frac{\left|g_1^{(i)}(x)\varphi^i(x)u(x)\right|^{q_{i1}}\left|g_2^{(i)}(x)\varphi^i(x)u(x)\right|^{q_{i2}}\cdots\left|g_n^{(i)}(x)
\varphi^i(x)u(x)\right|^{q_{in}}}{(i!)^{k_i}}.
\end{align*}
Then, taking into account \eqref{SeminormEstimater}, we have
\begin{equation*}
\left\|g_j^{(i)}\varphi^i u\right\|_{\infty} \leq \C_{i}\|g_j\|_{W^r_u}, \qquad i=1,2,\ldots,r, \quad j=1,2,\ldots,n,
\end{equation*}
and, in virtue of \eqref{SobolevNormMultiD}, we deduce
\begin{align*}
\left|(f \circ g)^{(r)}(x)\varphi^r(x)u^r(x)\right| & \leq \C \|f\|_{{\bm W}^r(\Omega)}\displaystyle\sum_0 \sum_1\sum_2 \cdots \sum_r \frac{r!}{\prod_{i=1}^r\prod_{j=1}^n  q_{ij}!} \\ & \times
\prod_{i=1}^r\displaystyle\frac{\C_i^{q_{i1}+q_{i2}+\ldots+q_{in}}\displaystyle \prod_{j=1}^n\left\|g_j\right\|_{W^r_u}^{q_{ij}}}
{(i!)^{k_i}}\\
& \leq \C \|f\|_{{\bm W}^r(\Omega)}\displaystyle\sum_0 \sum_1\sum_2 \cdots \sum_r \frac{r!}{\prod_{i=1}^r\prod_{j=1}^n  q_{ij}!} \\ & \times
 \prod_{j=1}^n\left\|g_j\right\|_{{W^r_u}}^{p_j}\prod_{i=1}^r\frac{\C_i^{k_i}}{(i!)^{k_i}}.
\end{align*}
Since $p_j \leq k \leq r$ for  $j=1,2,\ldots,n$, setting ${\bar \C}=\displaystyle \max_{i=1,2,\ldots,r}\C_i$ and multiplying and dividing by $k_1!k_2!\cdots,k_r!$, we can write
\begin{align*}
\left|(f \circ g)^{(r)}(x)\varphi^r(x)\right| & \leq \C{\bar \C}^r \|f\|_{{\bm W}^r(\Omega)} \prod_{j=1}^n\left\|g_j\right\|_{W^r_u}^{s_j}
\displaystyle \sum_0 \frac{r!}{k_1!k_2!\cdots,k_r!}\displaystyle \\ & \times \left[\prod_{i=1}^r\frac{1}{(i!)^{k_i}}\right]\sum_1\sum_2 \cdots \sum_r  \prod_{i=1}^r\frac{k_i!}{\prod_{j=1}^n  q_{ij}!}
\end{align*}
with the exponents $s_j$, $j=1,2,\ldots,n$, defined by \eqref{exponentp}.
Now, we observe that
\begin{align*}
\sum_i \frac{k_i!}{\prod_{j=1}^n q_{ij}!}&=\sum_{q_{i1}+q_{i2}+\ldots+q_{in}=k_i} \binom{k_i}{q_{i1}\, q_{i2}\,\ldots \, q_{in}} \\ & =(1+1+\ldots+1)^{k_i}=n^{k_i} \qquad i=1,2,\ldots,r,
\end{align*}
from which it follows that
\begin{eqnarray*}
\sum_1\sum_2 \cdots \sum_r  \prod_{i=1}^r\frac{k_i!}{\prod_{j=1}^n  q_{ij}!}& = &\sum_1 \frac{k_1!}{\prod_{j=1}^n q_{1j}!}
\sum_2 \frac{k_2!}{\prod_{j=1}^n q_{2j}!}\ldots\sum_r \frac{k_r!}{\prod_{j=1}^n q_{rj}!} \\
&=& n^{k_1+k_2+\ldots+k_r}=n^k\leq n^r
\end{eqnarray*}
and, consequently,
\begin{align*}
& \left|(f \circ g)^{(r)}(x)\varphi^r(x)u^r(x)\right|  \\ & \leq \C n^r \|f\|_{{\bm W}^r(\Omega)} \prod_{j=1}^n\left\|g_j\right\|_{W^r_u}^{s_j}
\displaystyle \sum_0 \frac{r!}{k_1!k_2!\cdots,k_r!}\displaystyle \prod_{i=1}^r\frac{1}{(i!)^{k_i}}\\
& = \C n^r \|f\|_{{\bm W}^r(\Omega)} \prod_{j=1}^n\left\|g_j\right\|_{W^r_u}^{s_j}
\displaystyle \sum_{k=1}^r B_{r,k}(1,1,\ldots,1) \\ & = \C n^r B_r \|f\|_{{\bm W}^r(\Omega)} \prod_{j=1}^n\left\|g_j\right\|_{W^r_u}^{s_j},
\end{align*}
where $\C$ is a postive constant independent of $f$ and $g$, i.e. the thesis \eqref{EstDerCompFun}.
\end{proof}

By combining \eqref{EstDerCompFun} with \eqref{favard} we can immediately deduce  an estimate for $E_m(f \circ g)_{u}$.\\

\begin{theorem}
Let $f:\Omega \rightarrow \RR$ with $\Omega$ open subset of $\RR^n$ and $g:(-1,1) \rightarrow \RR^n$ such that $\mathrm{Im}(g) \subseteq \Omega$. Moreover, let $u(x)=(1-x)^{\gamma}(1+x)^\delta$ be a Jacobi weight with $0 \leq \gamma,\delta <1$. If we assume that $f \in  {\bm W}^{r}(\Omega)$ and $g=\left(g_1,g_2,\ldots,g_n\right)$ with $g_j \in W^r_{\sqrt[r]u}$, for $j=1,2,\ldots,n$, then
\begin{equation*}
E_m(f \circ g)_{u} \leq \frac{\mathcal{C}}{m^r}  n^r B_r \|f\|_{{\bm W}^{r}(\Omega)} \prod_{j=1}^n \|g_j\|_{W^r_{\sqrt[r]u}}^{s_j}
\end{equation*}
where $\C=\C(r,\gamma,\delta)$ is a positive constant independent of $f$ and $g$, $B_r$ is the $r$-th Bell number, and the exponents $s_j$, $j=1,2,\ldots,n$, are given in \eqref{exponentp} (with $u$ replaced with $\sqrt[r]u$).
\end{theorem}\\

\begin{corollary}
	In the special case $n=1$ the previous estimate is simplified as follows
	\begin{equation*}
		E_m(f \circ g)_{u} \leq \frac{\mathcal{C}}{m^r} B_r \|f\|_{{W}^{r}(\Omega)} \|g\|_{W^r_{\sqrt[r]u}}^{s_1}.
	\end{equation*}
\end{corollary}

\end{document}